\documentclass[12pt]{amsart}

\usepackage{a4wide}

\usepackage{enumerate}

\newtheorem{theorem}{Theorem}[section]

\theoremstyle{definition}
\newtheorem{remark}[theorem]{Remark}

\newcommand{\C}{\ensuremath{\mathbb{C}}}
\newcommand{\R}{\ensuremath{\mathbb{R}}}
\renewcommand{\H}{\ensuremath{\mathbb{H}}}
\newcommand{\g}[1]{\ensuremath{\mathfrak{#1}}}
\newcommand{\cal}[1]{\ensuremath{\mathcal{#1}}}
\newcommand{\II}{\ensuremath{I\!I}}

\DeclareMathOperator{\codim}{codim}

\begin{document}
\title[Canonical extension of submanifolds and foliations]{Canonical extension of submanifolds and foliations\\ in noncompact symmetric spaces}

\author[M.~Dom\'{\i}nguez-V\'{a}zquez]{Miguel Dom\'{\i}nguez-V\'{a}zquez}
\address{Instituto Nacional de Matem\'atica Pura e Aplicada (IMPA), Rio de Janeiro, Brazil.}
\email{mvazquez@impa.br}

\thanks{The author has been supported by a fellowship at IMPA (Brazil), and projects EM2014/009, GRC2013-045 and MTM2013-41335-P
with FEDER funds (Spain)}



\begin{abstract}
We propose a method to extend submanifolds, singular Riemannian foliations and isometric actions from a boundary component of a noncompact symmetric space to the whole space. This extension method preserves minimal submanifolds, isoparametric foliations and polar actions, among other properties. One of the several applications yields the first examples of inhomogeneous isoparametric hypersurfaces in noncompact symmetric spaces of rank at least two.
\end{abstract}


\maketitle

\section{Introduction}

A proper isometric action of a Lie group $H$ on a Riemannian manifold $\bar{M}$ is \emph{polar} if there is an immersed connected submanifold $\Sigma$ intersecting all $H$-orbits and always orthogonally. Such submanifold $\Sigma$ turns out to be totally geodesic and is called a \emph{section} of the action. 

In this paper we consider a natural method to enlarge submanifolds from a section $\Sigma$ of a polar $H$-action on $\bar{M}$ to the whole ambient manifold $\bar{M}$, simply by attaching the $H$-orbits to the original submanifold. This procedure can also be applied to extend singular Riemannian foliations or isometric actions on $\Sigma$ to $\bar{M}$, under some mild assumptions.

Our first goal is the observation that, if the $H$-orbits are minimal submanifolds, then several important properties on submanifolds and foliations are preserved by this canonical extension method. Such properties include, for instance, submanifolds with parallel mean curvature, minimal submanifolds and isoparametric foliations.

Our second goal is to apply this method on symmetric spaces of noncompact type and rank higher than one. The reason why we can do this stems from the \emph{horospherical decomposition} of a noncompact symmetric space. More specifically, any given subset $\Phi$ of simple roots of a noncompact symmetric space $\bar{M}=G/K$ determines a totally geodesic submanifold $B_\Phi$ (which is itself a symmetric space of lower rank), an abelian subgroup $A_\Phi$ of $G$ and a nilpotent subgroup $N_\Phi$ of $G$ such that the solvable group $A_\Phi N_\Phi$ acts freely and polarly on $\bar{M}$ with section $B_\Phi$, and all orbits are minimal submanifolds and Einstein solvmanifolds. Thus, the extended examples in this setting are $A_\Phi N_\Phi$-equivariant.

Our work fits within the project of constructing examples of interesting geometric objects with ``lots of symmetry". In the context of submanifold theory, this project can be traced back at least to the influential work of Hsiang and Lawson~\cite{HL:JDG}. But more specifically, our results constitute a generalization of the canonical extension method for cohomogeneity one actions recently given by Berndt and Tamaru~\cite{BT:crelle}, and show that such canonical extension is of deeper nature and also works for isoparametric foliations and polar actions, among others. Moreover, when applied to isoparametric hypersurfaces, our method can be seen as a generalization of the construction of inhomogeneous isoparametric hypersurfaces in rank one noncompact symmetric spaces proposed by D\'iaz-Ramos and the author~\cite{DD:MZ}, \cite{DD:adv}. 

We present several concrete applications of our method. First, we produce new examples of inhomogeneous isoparametric families of hypersurfaces in symmetric spaces with restricted root system of $(BC_r)$-type, such as the complexified Cayley hyperbolic plane or the indefinite complex and quaternionic Grassmannians. Based on the existence of inhomogenous isoparametric foliations of higher codimension on real hyperbolic spaces, we construct inhomogeneous isoparametric foliations of codimension greater than one on the indefinite real Grassmannians. Our method also provides examples of polar non-hyperpolar actions (on noncompact symmetric spaces of rank higher than one) different from the ones introduced in~\cite{BDT:JDG}, and examples of minimal and constant mean curvature hypersurfaces of cohomogeneity one in many symmetric spaces.

This article is organized as follows. In Section~\ref{sec:extension} we prove the canonical extension theorem for polar actions with minimal orbits. Section~\ref{sec:horospherical} introduces the polar action associated with a horospherical decomposition of a noncompact symmetric space. Finally, in Section~\ref{sec:applications} we present several applications of our method.

\section{Canonical extension for polar actions with minimal orbits} \label{sec:extension}
In Theorem~\ref{th:polar} below we prove the canonical extension theorem for a polar action with minimal orbits. For simplicity, we additionally assume that the action is free and each orbit meets the section only once. With a view to the application to noncompact symmetric spaces, these assumptions are natural, as follows from Section~\ref{sec:horospherical}. However, Theorem~\ref{th:polar} remains true under slightly weaker conditions (see Remark~\ref{rem:Coxeter}).

\begin{theorem}\label{th:polar}
Let $\bar{M}$ be a Riemannian manifold and $H$ a connected group of isometries of $\bar{M}$ acting freely and polarly on $\bar{M}$ with section $\Sigma$. Assume that each $H$-orbit intersects $\Sigma$ exactly once, and is a minimal submanifold. 

Let $M$ be an immersed (resp.\ embedded) connected submanifold of codimension~$k$ in $\Sigma$, and $\mathcal{F}$ a singular Riemannian foliation of codimension $k$ on $\Sigma$. Consider the sets
\[
H\cdot M=\{h(p):h\in H, p\in M\} \qquad \text{and}\qquad H\cdot \mathcal{F}=\{H\cdot L: L\in \mathcal{F}\}.
\]
Then:
\begin{enumerate} [{\rm (i)}]
\item $H\cdot M$ is an immersed (resp.\ embedded) connected submanifold of codimension~$k$ in~$M$.
\item \label{item:II} The second fundamental form of $H\cdot M$ is $H$-equivariant and satisfies 
\[
\qquad\;\II_{H\cdot M}(X,Y)=\II_M(X,Y) , \quad \II_{H\cdot M}(U,V)=(\II_{H\cdot p}(U,V))^\perp \quad \text{and}\quad \II_{H\cdot M}(X,U)=0,
\]
for all $X,Y\in T_pM$, $U,V\in T_p(H\cdot p)$ and $p\in M$, where $(\cdot)^\perp$ denotes orthogonal projection onto the normal space of $H\cdot M$.
\item The mean curvature vector field of $H\cdot M$ is $H$-equivariant and, along $M$, coincides with that of $M$.
\item If $M$ has one of the following properties as a submanifold of $\Sigma$, so does $H\cdot M$ as a submanifold of $\bar{M}$: parallel mean curvature, minimal, (globally) flat normal bundle, isoparametric.
\item $H\cdot \mathcal{F}$ is a singular Riemannian foliation on $\bar{M}$ of codimension $k$.
\item If $\mathcal{F}$ has one of the following properties as a foliation on $\Sigma$, so does $H\cdot \mathcal{F}$ as a foliation on $\bar{M}$: polar, hyperpolar, isoparametric.
\item If $\cal{F}$ is the orbit foliation on $\Sigma$ of the action of a group of isometries $S$ of $\bar{M}$ leaving $\Sigma$ invariant and normalizing $H$, then $H\cdot \cal{F}$ is the orbit foliation of the isometric action of $HS=SH$ on $\bar{M}$.
\end{enumerate}
\end{theorem}

Before proving Theorem~\ref{th:polar}, we recall some relevant definitions.

Let $M$ be an immersed submanifold of a Riemannian manifold $\bar{M}$. We say that $M$ has globally flat normal bundle if every normal vector to $M$ can be extended to a parallel (with respect to the normal connection) normal vector field $\xi\in \Gamma(\nu M)$ defined globally on $M$; here and henceforth $\Gamma(\cdot)$ denotes the sections of a vector bundle. If each point of $M$ admits a neighbourhood with globally flat normal bundle, we say that $M$ has flat normal bundle.

The submanifold $M$ is \emph{isoparametric}, in the sense of Heintze, Liu and Olmos~\cite{HLO:chev}, if:
\begin{enumerate}[{\rm (a)}]
\item $M$ has flat normal bundle,
\item locally, the parallel submanifolds to $M$ in $\bar{M}$ have parallel mean curvature, i.e.\ for any $p\in M$ there is an open neighbourhood $U$ of $p$ in $M$ and a parallel normal field $\xi$ on $U$ such that, for $r>0$ small enough, the parallel submanifold $U^r=\{\exp_p(r\xi_p):p\in U\}$ has parallel mean curvature, and
\item $M$ admits sections, i.e.\ through any point in $M$ there is a totally geodesic submanifold $\Sigma'$ of $\bar{M}$ intersecting $M$ orthogonally and with $\dim \Sigma'=\codim M$.
\end{enumerate} 

A decomposition $\cal{F}$ of an ambient manifold $\bar{M}$ into  injectively immersed submanifolds (called leaves) is a \emph{singular Riemannian foliation}, in the sense of Molino (cf.~\cite{Th:milan}), if it is
\begin{enumerate}[{\rm (a)}]
\item a transnormal system, i.e.\ every geodesic intersecting one leaf orthogonally is perpendicular to all leaves it meets, and
\item a singular foliation, i.e.\ the module $\mathfrak{X}_{\mathcal{F}}\subset \Gamma(T\bar{M})$ of vector fields that are everywhere tangent to the leaves of $\mathcal{F}$ satisfies $T_p L=\{X_p:X\in\mathfrak{X}_{ \mathcal{F}}\}$ for each $L\in \mathcal{F}$ and $p\in L$.
\end{enumerate}
The leaves of maximal dimension are called regular, and the others are singular. The codimension of a singular Riemannian foliation is the codimension of its regular leaves. In this paper, \emph{we use the word foliation to mean singular Riemannian foliation}.

A foliation $\cal{F}$ on $\bar{M}$ is polar if through each $p\in \bar{M}$ there is an immersed submanifold $\Sigma'$ that intersects all the leaves of $\cal{F}$ and always perpendicularly. Again, $\Sigma'$ is totally geodesic, its dimension agrees with the codimension of $\cal{F}$ and is called a section of~$\cal{F}$. If all sections of a polar foliation are flat (in its induced metric), the foliation is called hyperpolar. A polar foliation is isoparametric if its regular leaves are isoparametric submanifolds.

\begin{proof}[Proof of Theorem~\ref{th:polar}]
The restriction of the action to the section, i.e.\ $H  \times \Sigma \to \bar{M}$, $(h,q)\mapsto h(q)$, is a diffeomorphism, since the action is polar, free and each orbit intersects $\Sigma$ once. Hence, the restriction of the action to the immersed (resp.\ embedded) submanifold $M$ gives an immersion (resp.\ embedding) of $H\times M$ into $\bar{M}$ with image $H\cdot M$. The tangent space to $H\cdot M$ at $h(q)$, with $h\in H$, $q\in M$, is given by the orthogonal decomposition $T_{h(q)}(H\cdot M)= T_{h(q)}(H\cdot q)\oplus T_{h(q)}h(M)= h_*(T_q (H\cdot q))\oplus h_*(T_q M)$. The normal space at $h(q)$ is then $\nu_{h(q)}(H\cdot M)=h_*(\nu_q^\Sigma M)$, where $\nu_q^\Sigma M$ is the normal space to $M$ as submanifold of $\Sigma$. Hence, the codimension of $M$ in $\Sigma$ equals the codimension of $H\cdot M$ in $\bar{M}$. This~shows~(i).

Let us denote by $\bar{\nabla}$ the Levi-Civita connection of $\bar{M}$, by $(\cdot)^\perp$ the orthogonal projection onto the normal bundle $\nu (H\cdot M)$ of $H\cdot M$, and by $\II_{H\cdot M}$ and $\II_M$ the second fundamental forms of the submanifolds $H\cdot M$ and $M$, respectively. Since $M\subset \Sigma$ and $\Sigma$ is totally geodesic in $\bar{M}$, the second fundamental forms of $M$ as submanifold of $\bar{M}$ and of $\Sigma$ coincide, and the Levi-Civita connection of $\Sigma$ agrees with $\bar{\nabla}$ when applied to vector fields tangent to $\Sigma$. If $X$, $Y\in\Gamma (T(H\cdot M))$ and $h\in H$, since $h$ is an isometry and its differential preserves the tangent and normal bundles of $H\cdot M$, we~have 
\[
h_*\II_{H\cdot M}(X, Y)=h_*(\bar{\nabla}_X Y)^\perp=(h_*(\bar{\nabla}_X Y))^\perp=(\bar{\nabla}_{h_*X}h_*Y)^\perp=\II_{H\cdot M}(h_*X, h_*Y),
\]
which shows the $H$-equivariance of $\II_{H\cdot M}$. Moreover, since $M\subset\Sigma$ and $\Sigma$ is totally geodesic, $\II_{M}$ and $\II_{H\cdot M}$ coincide when applied to vectors tangent to $M$. Now, if $X\in\Gamma(TM)$, and $U$, $V$ are vector fields on $H\cdot M$ tangent to the $H$-orbits, then at any point $p\in M$ we have $\II_{H\cdot M}(U,V)=(\II_{H\cdot p}(U,V))^\perp$, and $\II_{H\cdot M}(X,U)=(\bar{\nabla}_X U)^\perp=0$, since the shape operator of $\Sigma$ vanishes. This proves (ii).

Denote by $\mathcal{H}$ the mean curvature vector field of $H\cdot M$. Given $h\in H$ and $p\in M$, let $\{e_i\}$ and $\{f_j\}$ be orthonormal bases of $T_pM$ and $T_p(H\cdot p)$, respectively. Then the collection of the vectors $h_*e_i$ and $h_*f_j$ is an orthonormal basis of $T_{h(p)}(H\cdot M)$, and hence 
\[\textstyle
\mathcal{H}_{h(p)}=\sum_i\II_{H\cdot M}(h_*e_i, h_*e_i)+\sum_j\II_{H\cdot M}(h_*f_j,h_*f_j)=h_*\sum_i\II_{H\cdot M}(e_i, e_i)=h_*\mathcal{H}_p,
\]
since the orbit $H\cdot p$ is minimal. This proves (iii).

Now, for any submanifold $L$ of $\bar{M}$, we denote by $\nabla^{L,\perp}$ the normal connection of $L$ as a submanifold of $\bar{M}$. Similarly as above, since $\Sigma$ is totally geodesic, the normal connection of $M$ as a submanifold of $\Sigma$ agrees with $\nabla^{M,\perp}$. 

Assume that $M$ has parallel mean curvature, that is, $\nabla^{M,\perp}_X \mathcal{H}=0$, for all $X\in \Gamma(TM)$. Let $X\in \Gamma(TM)$, $Y\in \Gamma(H\cdot p)$ for some $p\in M$, and $h\in H$. Then, on the one hand, $\nabla^{H\cdot M,\perp}_{h_*X}\mathcal{H}=(\bar{\nabla}_{h_*X}h_*\mathcal{H})^\perp=h_*(\bar{\nabla}_{X}\mathcal{H})^\perp=h_*\nabla^{M,\perp}_X \mathcal{H}=0$. On the other hand, since $\mathcal{H}$ is an $H$-equivariant normal vector field to the $H$-orbits, which are all principal, and the action of $H$ is polar, it follows that  $\mathcal{H}$ is parallel with respect to the normal connection of any $H$-orbit (see~\cite[Corollary~3.2.5]{BCO}), and hence $\nabla^{H\cdot M,\perp}_{Y}\mathcal{H}=(\bar{\nabla}_{Y}\mathcal{H})^\perp=0$. Altogether we have shown that $\nabla^{H\cdot M,\perp}\mathcal{H}=0$.

The extension of minimal submanifolds follows directly from (iii).

Assume that $M$ has globally flat normal bundle. Given a generic point $h(p)$ of $H\cdot M$, with $h\in H$ and $p\in M$, any normal vector $\eta$ to $H\cdot M$ at $h(p)$ is of the form $\eta=h_*\xi$, for some $\xi\in\nu_p M\cap T_p\Sigma$. By assumption, we can extend $\xi$ to a parallel normal vector field along $M$, and then $H$-equivariantly to $H\cdot M$, that is, $\xi_{g(q)}=g_*\xi_q$, for any $g\in H$ and $q\in M$. This gives a well-defined normal vector field $\xi$ along $H\cdot M$. The same argument used above to show that $\nabla^{H\cdot M,\perp}\mathcal{H}=0$ proves that $\xi$ is parallel with respect to $\nabla^{H\cdot M,\perp}$. 

Now let $M$ be isoparametric in $\Sigma$. We already know that $H\cdot M$ has flat normal bundle. Then, let $\xi$ be a parallel normal vector field along $H\cdot U$, for some open subset $U$ of $M$. By the argument above, $\xi$ must be $H$-equivariant. Then
\begin{align*}
(H\cdot U)^r&=\{\exp_{h(p)}(r\xi_{h(p)}):h\in H,\,p\in U\}=\{\exp_{h(p)}(r h_*\xi_{p}):h\in H,\,p\in U\}
\\
&=\{h(\exp_{p}(r \xi_{p})):h\in H,\,p\in U\}=H\cdot U^r.
\end{align*}
Since $U^r$ has parallel mean curvature by assumption, its canonical extension $H\cdot U^r=(H\cdot U)^r$ has parallel mean curvature as well. It remains to show that $H\cdot M$ admits sections. Let $h(p)\in H\cdot M$, with $h\in H$, $p\in M$, and let $\Sigma'$ be a section for the isoparametric submanifold $M$ of $\Sigma$ through the point $p$. Since the action of $H$ is polar, $h(\Sigma')$ is a section for $H\cdot M$ through $h(p)$. All in all, $H\cdot M$ is isoparametric. This completes the proof~of~(iv).

Now we show that, if $\mathcal{F}$ is a foliation on $\Sigma$, then $H\cdot \mathcal{F}$ is a foliation on $\bar{M}$. By definition, the leaves of $H\cdot \mathcal{F}$ are the canonically extended submanifolds of the leaves of $\mathcal{F}$. We show first that $H\cdot \mathcal{F}$ is transnormal. Let $\gamma$ be a geodesic in $\bar{M}$ intersecting one leaf $H\cdot L\in H\cdot \mathcal{F}$ orthogonally at a point $h(p)$, with $L\in \mathcal{F}$, $h\in H$ and $p\in M$. Then $\gamma$ lies in $h(\Sigma)$ and $h^{-1}\circ\gamma$ is a geodesic in $\Sigma$ intersecting $L$ at $p$ perpendicularly. If $\gamma$ meets another leaf $H\cdot L'\in H\cdot \mathcal{F}$, with $L'\in \mathcal{F}$, then $h^{-1}\circ\gamma$ meets $L'$, and by the transnormality of $\mathcal{F}$, it does so perpendicularly, and hence, $\gamma$ intersects $H\cdot L'$ orthogonally.

We prove now that $H\cdot \mathcal{F}$ is a singular foliation. First, the module $\mathfrak{Y}$ of Killing vector fields associated with the $H$-action is a submodule of $\mathfrak{X}_{H\cdot \mathcal{F}}$ such that $T_{h(p)}(H\cdot p)=\{Y_{h(p)}:Y\in\mathfrak{Y}\}$ for all $h\in H$ and $p\in \Sigma$. Second, the module $\tilde{\mathfrak{X}}_\mathcal{F}$ made of the vector fields $h_*X$, for $X\in\mathfrak{X}_\mathcal{F}$ and $h\in H$, is a submodule of $\mathfrak{X}_{H\cdot \mathcal{F}}$ such that $T_{h(p)} (h(L))=\{X_{h(p)}:X\in\tilde{\mathfrak{X}}_\mathcal{F}\}$ for any $h\in H$, $L\in\mathcal{F}$ and $p\in L$. Hence $\mathcal{Y}\oplus\tilde{\mathfrak{X}}_\mathcal{F}\subset \mathfrak{X}_{H\cdot \mathcal{F}}$ and then
\begin{align*}
T_{h(p)} (H\cdot L)&= T_{h(p)} h(L)\oplus T_{h(p)}(H\cdot p)=\{X_{h(p)}+Y_{h(p)}: X\in\tilde{\mathfrak{X}}_\mathcal{F}, \, Y\in \mathfrak{Y}\}
\\
&\subset\{Z_{h(p)}: Z\in\mathfrak{X}_{H\cdot \mathcal{F}}\},
\end{align*}
which implies that $H\cdot \mathcal{F}$ is a singular foliation. This shows (v).

The extension of polar and isoparametric foliations follows from the arguments used to prove (iv), whereas the extension of hyperpolar foliations is obvious. This shows (vi).

Finally, let $S$ be a group of isometries of $\bar{M}$ that leaves $\Sigma$ invariant and normalizes $H$. Then the group $HS=SH=\{hs:h\in H, s\in S\}$ is a group of isometries of $\bar{M}$. Let $\cal{F}$ be the orbit foliation of the $S$-action on $\Sigma$. If $L\in\cal{F}$ and $h(p)\in H\cdot L$, with $p\in L$, $h\in H$, then for any $h'\in H$ and $s\in S$ there is an $h''\in H$ such that $h's(h(p))=h'h''s(p)\in H\cdot L$, so $HS\cdot h(p)\subset H\cdot L$. Moreover, if $h'(q)$ is another point of $H\cdot L$, with $q\in L$ and $h'\in H$, then there is an $s\in S$ such that $s(p)=q$, and hence $h'(q)=h's(p)=h'sh^{-1}(h(p)) = h'h''s(h(p))$ for some $h''\in H$, and hence $H\cdot L\subset HS\cdot h(p)$. This concludes the proof.
\end{proof}

\begin{remark}\label{rem:Coxeter}
In Theorem~\ref{th:polar} we assumed that the $H$-action is free, polar, and such that each orbit meets the section only once. One could relax  these assumptions by replacing the freeness condition by the assumption that all $H$-orbits are principal. Equivalently, one could require the $H$-action to be Coxeter polar (in the sense of Grove and Ziller~\cite{GZ}) with no singular orbits. Obviously, the minimality of the orbits remains an essential assumption and, thus, the absence of singular orbits is a natural condition, in view of~\cite[\S4]{AR}. 
\end{remark}

\section{The horospherical decomposition of a noncompact symmetric space and the associated polar action}\label{sec:horospherical}
In this section we review the construction of the horospherical decomposition of a symmetric space of noncompact type, and present the associated polar action; see~\cite{BT:crelle} for~details.

Let $\bar{M}$ be a connected Riemannian symmetric space of noncompact type. Then $\bar{M}$ is diffeomorphic to $G/K$, where $G$ is the identity connected component of the isometry group of $\bar{M}$ and $K$ is the isotropy group at a base point $o\in \bar{M}$. Let $\g{g}$ and $\g{k}$ be the Lie algebras of $G$ and $K$, respectively. Then $\g{g}$ is a real semisimple Lie algebra, whereas $\g{k}$ is a maximal compact subalgebra of $\g{g}$. Let $\g{g}=\g{k}\oplus\g{p}$ be the corresponding Cartan decomposition, where $\g{p}$ is the orthogonal complement of $\g{k}$ in $\g{g}$ with respect to the Killing form of $\g{g}$. 

Let $\g{a}$ be a maximal abelian subspace of $\g{p}$. For each covector $\alpha$ on $\g{a}$ we define the subspace $\g{g}_\alpha=\{X\in\g{g}:[H,X]=\alpha(H)X\text{ for all }H\in\g{a}\}$ of $\g{g}$. Each covector $\alpha$ such that $\alpha\neq 0$ and $\g{g}_\alpha\neq 0$ is called a restricted root, and every nonzero $\g{g}_\alpha$ is a restricted root space. Thus, we have the restricted root space decomposition $\g{g}=\g{g}_0\oplus \left( \bigoplus_{\alpha\in\Sigma}\g{g}_\alpha\right)$,
where $\Sigma$ is the set of restricted roots, $\g{g}_0=\g{k}_0\oplus\g{a}$, and $\g{k}_0$ is the centralizer of $\g{a}$ in $\g{k}$. 

Let $r$ be the rank of $\bar{M}$ and $\Lambda=\{\alpha_1,\dots,\alpha_r\}$ a set of simple roots of $\Sigma$. Denote by $\Sigma^+$ the corresponding set of positive roots. Then $\g{n}=\bigoplus_{\alpha\in\Sigma^+}\g{g}_\alpha$ is a nilpotent subalgebra, $\g{a}\oplus\g{n}$ is a solvable subalgebra, and $\g{g}=\g{k}\oplus\g{a}\oplus\g{n}$ is an Iwasawa decomposition of $\g{g}$.

The subalgebra $\g{k}_0\oplus\g{a}\oplus\g{n}$ is called a minimal parabolic subalgebra of $\g{g}$. A subalgebra $\g{q}$ of $\g{g}$ is called parabolic if it contains a minimal parabolic subalgebra of $\g{g}$ up to conjugation. It turns out that the conjugacy classes of parabolic subalgebras of $\g{g}$ are parametrized by the subsets $\Phi$ of the set $\Lambda$ of simple roots. Below we explain the explicit construction.

Let $\Phi\subset \Lambda$, denote by $\Sigma_\Phi$ the root subsystem of $\Sigma$ generated by $\Phi$, and put $\Sigma_\Phi^+=\Sigma_\Phi\cap\Sigma^+$. We define the following subalgebras of $\g{g}$:
\[\g{l}_\Phi=\g{g}_0\oplus\bigl(\bigoplus_{\alpha\in\Sigma_\Phi}\g{g}_\alpha\bigr), \quad \g{g}_\Phi=[\g{l}_\Phi,\g{l}_\Phi], \quad  \g{n}_\Phi=\bigoplus_{\alpha\in\Sigma^+\setminus\Sigma^+_\Phi}\g{g}_\alpha \quad \text{and}\quad \g{a}_\Phi=\bigcap_{\alpha\in\Phi}\ker\alpha,\]
which are, respectively, reductive, semisimple, nilpotent and abelian.
The centralizer and normalizer of $\g{a}_\Phi$ in $\g{g}$ is $\g{l}_\Phi$, and $[\g{l}_\Phi,\g{n}_\Phi]\subset \g{n}_\Phi$. Then, the Lie algebra $\g{q}_\Phi=\g{l}_\Phi\oplus\g{n}_\Phi$ is the  parabolic subalgebra of $\g{g}$ associated with the subset $\Phi$ of~$\Lambda$. 
As limit cases we obtain the minimal parabolic subalgebra $\g{q}_\emptyset=\g{k}_0\oplus\g{a}\oplus\g{n}$ if $\Phi=\emptyset$, and $\g{q}_\Lambda=\g{g}$ if $\Phi=\Lambda$.

Let $A$, $N$, $N_\Phi$ and $G_\Phi$ be the connected subgroups of $G$ with Lie algebras $\g{a}$, $\g{n}$, $\g{n}_\Phi$ and $\g{g}_\Phi$, respectively. The groups $AN=NA$ and $A_\Phi N_\Phi=N_\Phi A_\Phi$ are solvable subgroups of $G$ with Lie algebras $\g{a}\oplus\g{n}$ and $\g{a}_\Phi\oplus\g{n}_\Phi$. 
The orbit $B_\Phi=G_\Phi\cdot o$ is a connected totally geodesic submanifold of $\bar{M}$. $B_\Phi$ is itself a noncompact symmetric space whose rank agrees with the cardinality of $\Phi$, and the identity connected component of its isometry group is contained in $G_\Phi$. $B_\Phi$ is called the \emph{boundary component} of $\bar{M}$ corresponding to the choice of $\Phi\subset\Lambda$. 

Finally, the \emph{horospherical decomposition} of $\bar{M}=G/K$ is the analytic diffeomorphism
\[
A_\Phi\times N_\Phi\times B_\Phi \to \bar{M},\qquad (a, n, p)\mapsto an(p).
\]

It was proved by Berndt, D\'iaz-Ramos and Tamaru that the $A_\Phi N_\Phi$-action on $\bar{M}$ is polar with section $B_\Phi$, see~\cite[Proposition~4.2]{BDT:JDG}. It turns out that the orbits of this action are all congruent to each other by isometries in $G_\Phi$. Moreover, Tamaru~\cite{Ta:mathann} showed that these orbits are Einstein solvmanifolds and minimal submanifolds of $\bar{M}$. They are totally geodesic if and only if $\Phi$ and $\Lambda\setminus\Phi$ are orthogonal. All these results imply the following. 

\begin{theorem}\label{th:Phi}
The action of $A_\Phi N_\Phi$ on $\bar{M}$ is free and polar with section~$B_\Phi$. Each orbit is a minimal submanifold of $\bar{M}$ that intersects $B_\Phi$ exactly once.
\end{theorem}

\section{The canonical extension for noncompact symmetric spaces}\label{sec:applications}

In view of Theorem~\ref{th:Phi}, we can apply the canonical extension method given in Theorem~\ref{th:polar} to the polar action of a solvable Lie group $A_\Phi N_\Phi$, for each subset $\Phi$ of simple roots, on the corresponding symmetric space of noncompact type $\bar{M}=G/K$. Thus, we obtain a method to extend certain kinds of submanifolds and singular Riemannian foliations from each boundary component $B_\Phi$ to the whole $G/K$. Moreover, since the identity connected component of the isometry group of $B_\Phi$ is contained in $G_\Phi$, and $G_\Phi$  normalizes $A_\Phi N_\Phi$, it follows that isometric actions by connected groups can be extended from $B_\Phi$ to $G/K$. Below we present some applications of the canonical extension method in this context. 

\subsection{Inhomogeneous isoparametric hypersurfaces in spaces of type $(BC_r)$}
In \cite{DD:MZ} and \cite{DD:adv} many examples of inhomogeneous isoparametric families of hypersurfaces in the noncompact rank one symmetric spaces of nonconstant curvature were constructed. More precisely, we know only one example in the complex hyperbolic $3$-space $\C H^3$ and in the Cayley hyperbolic plane $\mathbb{O} H^2$, and uncountably many examples in the complex hyperbolic spaces $\C H^n$, $n\geq 4$, and in the quaternionic hyperbolic spaces $\H H^n$, $n\geq 3$. Except for the example in $\mathbb{O} H^2$, the corresponding hypersurfaces have nonconstant principal curvatures.

Whereas the Cayley hyperbolic plane does not appear as a proper boundary component of any irreducible symmetric space, the complex and quaternionic spaces do appear as boundary components $B_\Phi$ in some symmetric spaces $G/K$ with restricted root system of $(BC_r)$-type, when one takes $\Phi=\{\alpha_r\}$, where $\alpha_r$ is the shortest simple root; see \cite[pp.~119 and~146]{Loos}. Thus, $B_\Phi=\C H^{n+1}$ if $G/K$ is the indefinite complex Grassmannian $SU_{r+n,r}/S(U_{r+n} U_r)$, $n\geq 1$, $r\geq 2$, $B_\Phi=\H H^{n+1}$ if $G/K$ is the indefinite quaternionic Grassmannian $Sp_{r+n,r}/Sp_{r+n} Sp_r$, $n\geq 1$, $r\geq 2$, $B_\Phi=\C H^3$ for $SO_{2r+1}(\H)/U_{2r+1}$, $r\geq 2$, and $B_\Phi=\C H^5$ for the complexified Cayley hyperbolic plane $E_6^{-14}/SO_2 SO_{10}$.

Therefore, the extension of the examples in \cite{DD:MZ} and \cite{DD:adv} yields inhomogeneous isoparametric families of hypersurfaces with nonconstant principal curvatures in these higher rank spaces. For instance, in $E_6^{-14}/SO_2SO_{10}$,  in $SU_{r+n,r}/S(U_{r+n} U_r)$ and in $Sp_{r+n,r}/Sp_{r+n} Sp_r$, for $n\geq 3$, $r\geq 2$, we obtain uncountably many noncongruent examples, whereas in $SU_{r+2,r}/S(U_{r+2} U_r)$ and in $SO_{2r+1}(\H)/U_{2r+1}$, $r\geq 2$, we just obtain one example.

The extended examples admit a direct description as follows. Let $\alpha$ be any simple root in $\Lambda$ and $\g{w}$ any proper subspace of $\g{g}_{\alpha}$. Consider the connected subgroup $S_{\g{w}}$ of $AN$ with Lie algebra $\g{a}\oplus\bigl(\bigoplus_{\lambda\in\Sigma^+\setminus\{\alpha\}}\g{g}_\lambda\bigr)\oplus\g{w}$. The orbit of $S_{\g{w}}$ through the base point $o\in G/K$ and the distance tubes around it define an isoparametric family of hypersurfaces in $G/K$. This family is the canonical extension of an isoparametric family in $B_\Phi$, with $\Phi=\{\alpha\}$. This description extends in a natural way the construction proposed in \cite{DD:MZ} and \cite{DD:adv} to higher rank symmetric spaces of noncompact type. Now, when $G/K$ is of $(BC_r)$-type and $\alpha=\alpha_r$ is the shortest root in $\Lambda$, for many choices of the subspace $\g{w}\subset \g{g}_{\alpha_r}$, the resulting hypersurfaces are inhomogeneous with nonconstant principal curvatures; see \cite{DD:MZ} and \cite{DD:adv} for details. 

\subsection{Inhomogeneous isoparametric foliations on indefinite real Grassmannians}
Wu~\cite{Wu:tams} reduced the classification problem of isoparametric foliations on real hyperbolic spaces $\R H^n$ to the corresponding problem in spheres. In particular, given a geodesic sphere in $\R H^n$ (which is isometric to a standard sphere $S^{n-1}$), we can foliate it by an isoparametric foliation, and then extend it to the whole $\R H^n$ by exponentiating in directions normal to the geodesic sphere. This foliation is isoparametric in $\R H^n$. If one starts with an inhomogeneous foliation, the resulting foliation on $\R H^n$ is inhomogeneous as well. The only known examples of inhomogeneous (irreducible) isoparametric foliations on spheres $S^{n-1}$ are the ones of codimension one constructed by Ferus, Karcher and M\"unzner for $n-1\equiv 3\, (\mathrm{mod}\, 4)$ and $n\geq 16$; see~\cite{Th:milan} and \cite{FKM}. This yields examples of inhomogeneous isoparametric foliations of codimension two on $\R H^n$, for $n\geq 16$. 

The only irreducible noncompact symmetric spaces of rank at least two that admit a real hyperbolic space of dimension at least $16$ as boundary component are the indefinite real Grassmannians $SO^0_{r+n,r}/SO_{r+n} SO_r$, for $n\geq 15$ and $r\geq 2$; again, see~\cite{Loos}. In this case, $B_\Phi=\R H^{n+1}$, for $\Phi=\{\alpha_r\}$, where $\alpha_r$ is the shortest simple root. The canonical extension of the inhomogeneous isoparametric foliations on $\R H^n$ mentioned above produces examples of inhomogeneous isoparametric foliations of codimension two on indefinite real Grassmannians of sufficiently high dimension. Contrary to the examples in $\R H^n$, which were reducible, the extended foliations do not have any totally geodesic leaf.

\subsection{Polar non-hyperpolar actions}
Recently, Kollross and Lytchak~\cite{KL} concluded that polar actions on irreducible symmetric spaces of compact type and rank at least two are hyperpolar. This result is false in the noncompact setting, as shown in \cite[Proposition~4.2]{BDT:JDG}: the group $V\times N_\Phi$, where $V$ is any subgroup of $A_\Phi$ and $\Phi$ is a non-empty subset of simple roots, acts polarly, but not hyperpolarly, on the corresponding noncompact symmetric~space. 

However, the canonical extension method allows to enlarge any polar non-hyperpolar action on a rank one symmetric space arising as a boundary component of a higher rank space $G/K$ to a polar non-hyperpolar action on $G/K$. This is so because any section of the original action is also a section of the extended action. Moreover, there are many orbit inequivalent polar actions of cohomogeneity at least two on rank one noncompact symmetric spaces; see~\cite{Wu:tams} and~\cite{DDK} for classifications in $\R H^n$ and $\C H^n$, and \cite{Ko:Lie} for examples in $\H H^n$ and $\mathbb{O} H^2$. Thus, we get many new examples of polar non-hyperpolar actions on many higher rank noncompact symmetric spaces. Just observe that we do not get any new example if we extend a polar action on $\R H^2$, but we do get new examples for those higher rank symmetric spaces admitting some simple root space with dimension at least two.

\subsection{Minimal and CMC hypersurfaces}
The canonical extension method allows to obtain many examples of minimal and constant mean curvature (CMC) hypersurfaces in all symmetric spaces of noncompact type and rank at least two by extending minimal and CMC hypersurfaces from any boundary component isometric to a hyperbolic space. When such boundary component is a real hyperbolic plane $\R H^2$, the extended examples are known; see~\cite[Section~6]{BDT:JDG}. However, if the boundary component has at least dimension three, then it admits many nontrivial examples of minimal and CMC hypersurfaces. For instance, one can restrict to cohomogeneity one examples, that is, to hypersurfaces admitting an isometric action of cohomogeneity one (as in~\cite{DD:tams} for real hyperbolic spaces, or in~\cite{GG} for complex hyperbolic spaces). In this particular case, the extension method produces new examples of cohomogeneity one minimal or CMC hypersurfaces in any symmetric space of noncompact type admitting a simple root space of dimension at least two.

\enlargethispage{1em}


\end{document}